\documentclass{amsart}

\usepackage{amsmath,amscd,amsfonts,amssymb,amsthm, multirow, newlfont,exscale,color,mathrsfs,stmaryrd, multicol}

\title{Faithfulness of Top Local Cohomology\\ Modules in Domains}
\author{Melvin Hochster$^1$ and Jack Jeffries$^2$}
\address{Department of Mathematics \\ University of Michigan \\ Ann Arbor, MI 48109--1043, USA}
\email{hochster@umich.edu}
\address{Matem\'aticas B\'asicas Group \\ Centro de Investigaci\'on en Matem\'aticasl \\ Guanajuato, Gto. 36023, M\'exico}
\email{jeffries@cimat.mx}

\date{\today}
\usepackage{enumerate}
\usepackage{bm}
\theoremstyle{plain}
\newtheorem{theorem}{Theorem}[section]

\newtheorem{corollary}[theorem]{Corollary}

\newtheorem{proposition}[theorem]{Proposition}
\newtheorem{lemma}[theorem]{Lemma}
\newtheorem{conjecture}[theorem]{Conjecture}
\theoremstyle{remark}
\newtheorem{remark}[theorem]{Remark}

\theoremstyle{definition}

\newtheorem{question}[theorem]{Question}

\newcommand{\inc}{\subseteq}

\newcommand{\fA}{\mathfrak{A}}

\newcommand{\fm}{\mathfrak{m}}

\newcommand{\Hom}{\mathrm{Hom}}
\newcommand{\Ker}{\mathrm{Ker}}

\newcommand\rmk{(R,\,\fm,\,K)}
\newcommand{\Spec}{\textrm{Spec}}
\newcommand{\surj}{\twoheadrightarrow}

\newcommand{\mx}{\begin{pmatrix}}
\newcommand{\emx}{\end{pmatrix}}

\def\todo#1
{\par\vskip 7 pt \noindent {\textcolor{yellow}\footnotesize \color{yellow} \fbox{\parbox{4.8in}{\color{black}
\textbf{To do: } {\color{blue}#1}}}} \vskip 3 pt \par}
\def\forth#1
{\par\vskip 7 pt \noindent {\textcolor{yellow}\footnotesize \color{yellow} \fbox{\parbox{4.8in}{\color{black}
\textbf{For thought: } {\color{blue}#1}}}} \vskip 3 pt \par}

\newcommand{\cdim}{\mathrm{cd}}

\newcommand{\Ann}{\mathrm{Ann}}

\begin{document}

\begin{abstract} We study the conditions under which the highest nonvanishing local cohomology module of a domain $R$
with support in an ideal $I$ is faithful over $R$, i.e., which guarantee that $H^c_I(R)$ is faithful, where $c$ is the cohomological dimension
of $I$.   In particular, we prove that this is true for the case of positive prime characteristic
when $c$ is the number of generators of $I$.  \end{abstract}

\subjclass[2000]{Primary 13}


\thanks{$^1$The first author was partially supported by  National Science Foundation grants
DMS--1401384 and DMS--1902116.}
\thanks{{$^2$The second author was partially supported by National Science Foundation grant DMS--1606353.}}

\maketitle

\pagestyle{myheadings}
\markboth{MELVIN HOCHSTER AND JACK JEFFRIES}{FAITHFULNESS OF TOP LOCAL COHOMOLOGY MODULES}

\section{Introduction}\label{intro}

Throughout, all rings are commutative, Noetherian, associative with identity, and {\it local} ring $\rmk$ means Noetherian ring $R$
with unique maximal ideal $\fm$ and residue class field $R/\fm = K$.   

The \emph{local cohomology} functors with support in an ideal $I$ of $R$ are defined as $H^i_I(-):=\varinjlim_n \mathrm{Ext}^i_R(R/I^n,-)$. The vanishing or nonvanishing of the modules $H^i_I(R)$ is related to many other interesting algebraic and geometric properties of $R$ and $I$. For example,
\begin{itemize}
	\item The least $i$ for which $H^i_I(R)\neq 0$ is the depth of $I$ on $R$ \cite[Theorem~9.1]{24Hours};
	\item The largest $i$ for which $H^i_{\fm}(R)\neq 0$ in a local ring $(R,\fm)$ is the dimension of $R$  \cite[Theorem~9.3]{24Hours};
	\item If $(R,\fm)$ is a complete local domain, and $H^{i}_{I}(R)=0$ for $i>\dim(R)-2$, then $\mathrm{Spec}(R) \smallsetminus V(I)$ is connected \cite[Theorem~15.11]{24Hours};
	\item For $R=\mathbb{C}[x_1,\dots,x_n]$, if $H^i_I(R)=0$ for all $i>t$, then we have that, for all $i>t$, $H^{n+i}((\mathbb{C}^n \smallsetminus V(I))^{\mathrm{an}})=0$,  where ${\underline{\;\;}}^{\mathrm{an}}$ denotes the associated analytic space \cite[Theorem~19.25]{24Hours};
	\item The \emph{cohomological dimension of $I$}, the largest $i$ for which $H^i_I(R)\neq 0$, is a lower bound for the \emph{arithmetic rank} of $I$, the minimal number of generators of $I$ up to radical  \cite[Proposition~9.12]{24Hours}. 
\end{itemize}

Related to the question of vanishing is the study of annihilators of local cohomology. In this paper, we study the following question, which we state in two equivalent forms:

\begin{question}\label{mainqu}
\begin{enumerate}[(a)]
\item\label{qu-a} If $R$ is a domain that contains a field, $I$ is an ideal of $R$, and $c$ the cohomological dimension of $I$, must  $H^c_I(R)$ be a faithful $R$-module?
\item\label{qu-b} If $R$ contains a field, $I$ is an ideal of $R$, and $c$ the cohomological dimension of $I$, must the annihilator of $H^c_I(R)$ have height zero?
\end{enumerate}
\end{question}

This question is inspired by a conjecture of Lynch \cite{Ly1, Ly2}, which posits that if  $c$ is the cohomological dimension of the ideal $I$ of a local ring  
$\rmk$, and $J$ is the annihilator of the local cohomology module $H^c_I(R)$, then $R/J$ has the same Krull dimension as $R$. A number of positive results on Lynch's conjecture, that is cases where Question~\ref{mainqu} has an affirmative answer, have been established, including that it holds for rings of dimension at most three. We refer the reader to \cite{BoEg} for a summary of some of these results. 

However, Lynch's conjecture is false. The first counterexample to this was given by Bahmanpour~\cite{Bah}; this example is a nonequidimensional algebra over a field of arbitrary characteristic. A nonequidimensional counterexample in dimension three appears in~\cite{SiWa}. A counterexample to Lynch's conjecture in a power series ring over a DVR of mixed characteristic is given in forthcoming work of Datta, Switala, and Zhang \cite{DSZ}. We note here that for regular rings of characteristic zero \cite{Lyu1} and for strongly F-regular rings of positive characteristic \cite{BoEg} (hence for all regular rings containing a field), {every} nonzero local cohomology module is faithful.

In this note, we answer Question~\ref{mainqu} affirmatively in two main cases:
\begin{enumerate}[(i)]
	\item\label{case-1} $\mathrm{char}(R)=p>0$, and $\mathrm{cd}(I)=\mathrm{ara}(I)$;
	\item\label{case-2} $R$ is pure in a regular ring containing a field.
\end{enumerate} 
Notably, local cohomology modules in case~(\ref{case-1}) above  have closed support \cite{Kat}, but may have infinitely many 
associated primes \cite{SiSw}.

We also show that cases where there is an affirmative answer to Question~\ref{mainqu} imply a persistence property for cohomological dimension; see \S\ref{perscd} and, in particular, Corollary~\ref{pers.cor}.

\section{Main results}\label{main} 

For an ideal $I$ in a ring $T$ and a $T$-module $M$, we denote the cohomological dimension of $I$ with support in $M$ as
\[ \mathrm{cd}(I,M):= \sup\{ n \in \mathbb{N} \ | \ H^n_I(M)\neq 0 \}.\]
To prepare for the proof of the main theorem, we record a couple of lemmas that are likely known to experts.

\begin{lemma}\label{cd-Ass}
	Let $T$ be a Noetherian ring, and $I$ be an ideal. \begin{align*}
	\cdim(I,T)  &= \max\{ \cdim(I,M) \ | \ M \text{ is an } R\text{-module}\} \\
	&=\max\{\cdim(I, T/Q) \ | \ Q\in \mathrm{Min}(T)\}.
	\end{align*}
\end{lemma}
\begin{proof}
	The first equality is standard; see \cite[Theorem~9.6]{24Hours}.  For the second, let $c=\cdim(I,T)$. By the first equality, we have $\cdim(I,T)\geq \cdim(I,T/P)$ for all $P\in \Spec(R)$, and $\cdim(I,T/P)\geq \cdim(I,T/Q)$ if $P\subseteq Q$. Take a prime filtration $\{T_i\}$ of $T$. From the long exact sequence and the first equality we get right-exact sequences
	\[ H^c_I(T_i) \to H^c_I(T_{i+1}) \to H^c_I(T/Q_i) \to 0, \qquad Q_i\in \Spec(R), \]
	for each $i$. If $H^c_I(T/P)=0$ for every minimal prime of $T$, then $H^c_I(T/Q_i)=0$ for all $i$, and inductively we find that $H^c_I(T)=0$, a contradiction.
	\end{proof}
	
The equivalence of the statements~(\ref{qu-a}) and~(\ref{qu-b}) of Question~\ref{mainqu} follows easily from the previous lemma.

\begin{proposition}\label{prop-equiv}
Question~\ref{mainqu}~(\ref{qu-a}) and Question~\ref{mainqu}~(\ref{qu-b}) are equivalent.
\end{proposition}
\begin{proof}
If (\ref{qu-b}) has an affirmative answer, then clearly (\ref{qu-a}) does as well, since the only height zero ideal in a domain is the zero ideal. If (\ref{qu-a}) has an affirmative answer, let $c=\mathrm{cd}(I,R)$, and let $r\in \mathrm{Ann}_R(H^c_I(R))$. Then by Lemma~\ref{cd-Ass}, we have $H^c_I(R/P)\neq 0$ for some $P\in \mathrm{Min}(R)$, and the image of $r$ in $R/P$ annihilates $H^c_I(R) \otimes_R R/P \cong H^c_I(R/P)$. Then, by assumption, the image of $r$ is zero in $R/P$, so $\mathrm{Ann}_R(H^c_I(R))\subseteq P$, and consequently the annihilator has height zero.
\end{proof}

See also \cite[Proposition~3.1]{Bah} for another equivalent version of the conjecture. The following lemma is a form of local duality. Note that we are not restricting to finitely generated modules in the statement below.

\begin{lemma}\label{local-duality} Let $(A,\fm,k)$ be a complete Gorenstein local ring of dimension $d$. Let $E = H^d_{\fm}(A)$,
which is an injective hull of $k = A/\fm$ over $A$, and let $(-)^{\lor}=\Hom_A(-,E)$ be the Matlis duality functor. 
	
	Then, there is a natural isomorphism $\mathrm{Ext}^i_A(M,A) \cong H^{d-i}_{\fm}(M)^{\lor}$ for all $A$-modules $M$ and all $i=0,\dots,d$.
\end{lemma}
\begin{proof}
First, we recall that if $A$ is Gorenstein, then the \v{C}ech complex shifted by $d$ gives a flat resolution of $H^d_{\fm}(A)\cong E$. 
Using this to compute Tor gives isomorphisms $H^{d-i}_{\fm}(M) \cong \mathrm{Tor}_i^A(M,E)$. Applying Matlis duality yields
	\[ H^{d-i}_{\fm}(M)^{\lor} \cong \mathrm{Tor}_i^A(M,E)^{\lor} \cong \mathrm{Ext}^i_A(M,E^{\lor}) \cong \mathrm{Ext}^i_A(M,A), \]
	where the second isomorphism is \cite[Example~3.6]{Hun}.
\end{proof}

\begin{lemma}\label{Annihilator-Hom} Let $(A,\fm_A)\to (T,\fm_T)$ be a local homomorphism of complete local domains. Assume that $A$ is Gorenstein of dimension $d$. Then, \[\Ann_T(H^d_{\fm_A}(T))=\bigcap\limits_{\phi\in \Hom_A(T,A)} \Ker(\phi).\]
\end{lemma}
\begin{proof}
	By Lemma~\ref{local-duality}, there is an isomorphism $\Hom_A(T,A)\cong H^d_{\fm_A}(T)^\lor$, where $(-)^\lor=\Hom_A(-,E_{A}(A/\fm_A))$ is Matlis duality for $A$-modules.  By faithful exactness of the functor $(-)^\lor$, the map induced by multiplication by $t\in T$ annihilates $\Hom_A(T,A)$ if and only if it annihilates $H^d_{m_A}(T)$. Now, 
	\[\Ann_T(\Hom_A(T,A)) = \{ t\in T  \ | \ \forall \phi\in \mathrm{Hom}_A(T,A) , \, \phi(tT)=0 \}.\] If $\phi(t)=0$ for all $\phi\in \mathrm{Hom}_A(T,A)$, then $(\phi\circ \cdot t')(t)=\phi(t t')=0$ for all $t'\in T$, so $\phi(tT)=0$ as well. The stated equality follows.
\end{proof}

\begin{lemma}\label{reductions} Let  $R$ be a Noetherian domain, $I \subseteq R$ an ideal of $R$ such that 
$\mathrm{cd}(I,R)=\mathrm{ara}(I)$, and denote this value by $c$. Suppose that $H^c_I(R)$ is not faithful.  Then there
is an injective homorphism  $R \to S$,  where $S$ is a complete local domain with algebraically closed residue class field, such
that $H^c_{IS}(S) \neq 0$ and is not faithful over $S$.

Moreover, if  $R$ is equicharacteristic, we may choose a coefficient field $K \inc  S$, we may choose  $f_1, \,\ldots, \, f_c \in IS$
that generate $I$ up to radicals, and we may map the formal power series ring  $A:=K\llbracket x_1, \, \ldots,  \, x_c\rrbracket$ continuously to $S$ 
so that the map on $K$ is its inclusion in $S$ as coefficient field and
$x_i \mapsto f_i$.  This map is automatically injective, and we may identify  $K\llbracket x_1, \, \ldots,  \, x_c\rrbracket$ with its image
$K\llbracket f_1, \, \ldots,  \, f_c\rrbracket \subseteq S$. Once this identification is made, we have that $(x_1, \, \ldots,  \, x_c)S$ is an
ideal of cohomlogical dimension $c$ in $S$, while $H^c_{(x_1, \, \ldots,  \,x_c)}(S)$ is not a faithful  $S$-module.
	 \end{lemma}
\begin{proof}
  Suppose that there is some ideal $I=(f_1,\dots,f_c)$ such that $H:=H^{c}_I(R)\neq 0$, and there is some $x\neq 0$ such that $x H=0$. We can localize at a minimal prime ideal in the support of $H$, which necessarily contains $x$, to obtain a local choice of $S$ with
  $R \subseteq S$. 
 We can then complete at the maximal ideal of $S$: by faithful flatness, we have $H^c_{I\widehat{S}}(\widehat{S})\cong H^c_I(\widehat{S}) \cong H^c_I(R)\otimes_S \widehat{S}\neq 0$, and the image of $x$ annihilates this module. Note also that since all elements of $S$ are nonzerodivisors on $\widehat{S}$, and so do not lie in any of its associated primes.  The completion $\widehat{S}$ might no longer be a domain, but by Lemma~\ref{cd-Ass} above, for some $Q\in \mathrm{Min}(\widehat{S})$, we have $H^c_I(\widehat{S}/Q)\neq 0$, $R$ injects into $S$,  which injects into   $\widehat{S}/Q$, and $x$ annihilates $H^c_I(\widehat{S}/Q)$.  Consequently, $\widehat{S}/Q$
 is a new choice of $S$ that is a complete local domain. 

Now assume that $S$ is a complete local domain.  Fix a coefficient field for $S$, which we will denote by the same letter $K$ as the 
residue field. We now want to reduce to the case where 
$K$ is algebraically closed. We can take a faithfully flat local extension of $R$ with residue field $\overline{K}$; this extension again may not be a domain, but we may pass to the quotient by an associated prime and still have a counterexample, by the same argument as above.  Consequently, we have a choice of $S$ that is a complete local domain with algebraically closed residue field $K$.
 
 Let $(A,\fm_A,K)=(K\llbracket x_1,\dots, x_c\rrbracket,(x_1,\dots,x_c),K)$ be a power series ring over $K$, and consider the map 
 $\varphi:A\to R$ described in the statement of the lemma. The hypothesis on $H$ implies that $\varphi$ is injective: otherwise, $\varphi$ would factor through a local ring $(\overline{A},\fm_{\overline{A}})$ of dimension less than $c$, and, by Lemma~\ref{cd-Ass}, \[\cdim(I,R)=\cdim(\fm_{\overline{A}}R,R)=\cdim(\fm_{\overline{A}}, R) \leq \cdim(\fm_{\overline{A}},\overline{A})=\dim(\overline{A})<c.\]
 We may therefore identify the power series ring $A$ with its image in $R$.\end{proof}
 
 We are now ready to prove one of our main results. We refer the reader to  \cite{Ho} for basic properties of solid algebras over a domain.
 
 \begin{theorem}\label{maintheorem}
		 Let $R$ be a Noetherian domain of characteristic $p>0$. Let $I$ be an ideal of $R$ such that $\mathrm{cd}(I,R)=\mathrm{ara}(I)$, and denote this value by $c$. Then, $\Ann_R(H^{c}_I(R))=0$. 
	\end{theorem}

\begin{proof} As in the conclusion of Lemma~\ref{reductions}, we may assume that $R$ is a complete local domain with algebraically closed residue field $K$, and that there is a power series subring $A=K\llbracket x_1,\dots,x_c\rrbracket \subseteq R$ such that $I=\fm_A R$.

Since $H^c_{\fm_A}(R)\neq 0$, it follows from Lemma~\ref{Annihilator-Hom} that $\Hom_A(R,A)\neq 0$; i.e., $R$ is a solid $A$-algebra. Let $J$ be the annihilator of $H^c_{I}(R)$. We want to show that $J = 0$; suppose otherwise, to obtain a contradiction.
 
By definition of $J$, the intersection of the kernels of the $A$-linear maps from $R/J$ to $A$ is trivial, so there is an $A$-linear embedding
\[ R/J \hookrightarrow \!\!\!\! \prod\limits_{\phi\in \Hom_A(R/J,A)} \!\!\!\! A \qquad \qquad r \mapsto (\phi(r))_{\phi}. \]
 Let $P$ be a minimal prime of $J$.  Then $R/P$ embeds in
$R/J$ as an $R$-module, and so there is an $A$-linear embedding of $R/P$ into a product of copies of $A$.

 If $q = p^e$ and we replace  $A$ by $A^q$,  the inclusion $A^{q} \to R$ again satisfies the hypotheses of Lemma~\ref{Annihilator-Hom}, so
\[J=\bigcap\limits_{\phi\in \Hom_A(R,A)} \Ker(\phi) = \bigcap\limits_{\psi\in \Hom_{A^q}(R,A^q)} \Ker(\psi). \]
In particular, there is an $A^q$-linear embedding of $R/P \hookrightarrow \prod A^q$ into a product of copies of~$A^q$.

Let  $J'= J R_P \cap R$ be the $P$-primary component of $J$.  Choose $h$ such that $(P R_P)^h \subseteq J R_P$, so $P^{(h)}$ is properly contained in $J'$.  Then $P^{(h)}$ cannot contain $J$, since this would yield a contradiction after localizing at $P$.   Choose $u \in J \smallsetminus P^{(h)}$. 

Choose $q = p^e$ so that  $P^{[q]} \inc P^{(h)}$.  Consider $R/P^{(h)}$ as an $R^q$ module.  Then the elements of the
prime ideal $P^q$ of $R^q$ (i.e., the set of $q\,$th powers of elements in $P$) annihilate $R/P^{(h)}$, and so $R/P^{(h)}$ may be viewed as an  $R^q/P^q\cong (R/P)^q$-module. In fact, $R/P^{(h)}$ is a torsion-free $(R/P)^q$-module: if $\overline{r^q} \in (R/P)^q$, $\overline{s}\in R/P^{(h)}$, and $\overline{r^q} \cdot \overline{s} =0$, then $r^q s \in P^{(h)}$ in $R$, so either $s\in P^{(h)}$ (so that $\overline{s}=0$), or else $r^q\in P$ (so that $\overline{r^q}=0$ is zero) by primariness of $P^{(h)}$. Since $R$ is complete local with an algebraically closed residue field, it is F-finite, and the images of a finite generating set for $R$ as an $R^q$-module yield a finite generating set for $R/P^{(h)}$ as an $(R/P)^q$-module.

Hence,  $R/P^{(h)}$  embeds $(R/P)^q$-linearly
in a finitely generated free $(R/P)^q$-module.  Consequently, the image $v$ of $u$ in $R/P^{(h)}$ has nonzero coordinate projection
in some copy of $(R/P)^q$.  The composition  $R \surj R/P^{(h)} \to (R/P)^q$ gives an $A^q$-linear map such that the
image of $u$ is not 0.  Since $(R/P)^q$ embeds in a product of copies of $A^q$,  further composition gives an $A^q$-linear
map $R \to A^q$ such that the image of $u \in J $ is nonzero.  This is a contradiction. 
\end{proof}

\begin{corollary}\label{maincor}
Let $R$ be a Noetherian ring of characteristic $p>0$, and $I$ an ideal for which $\mathrm{cd}(I,R)=\mathrm{ara}(I)=c$. Then $\Ann_R(H^c_I(R))$ has height zero.
\end{corollary}
\begin{proof}
This is immediate from Theorem~\ref{maintheorem} and Proposition~\ref{prop-equiv}.
\end{proof}

We do not know whether the analogues of Theorem~\ref{maintheorem} and Corollary~\ref{maincor} hold in equal characteristic zero.
By extending results of Huneke, Katz, and Marley \cite{HuKM}, we may reduce this to the case of ideals with at most three generators.
This is immediate from:

\begin{proposition} A local cohomology module  $H^n_I(M)$ of a module $M$ over a Noetherian ring $R$ with support in 
an $n$-generated ideal for
$n \geq 4$ is isomorphic with a local cohomology module $H^3_J(M)$ where $J$ has at most three
generators.  
\end{proposition}
\begin{proof} It suffices to show that if $n  \geq  4$, we can reduce the number of generators and the cohomological index by 1.  Let $u,\,v,\,x,\,y$ be four
of the generators and $\fA$ the ideal generated by the $n-4$ remaining generators of $I$.  Let $J = (xu, yv, xv+yu, \fA)R$. It
suffices to show that $H^{n-1}_J(M) = H^{n}_I(M)$.  Note that  $xv, yu$ are roots of  
\[ z^2 -(xv+yu)z + (xv)(yu)  = 0 \]
and so are integral over $J$ and in its radical.  It follows that $(u,v)(R/\fA) \cap (x,y)(R/\fA)$, which has the same radical
as $(u,v)(x,y)(R/\fA)$, also has the same radical as $(xu, yv, xv+yu)(R/\fA)$. Hence, up to radicals,  $J$ is the intersection of $I_1 = (u,v)R + \fA$
and  $I_2 = (x,y)R + \fA$.  The Mayer-Vietoris sequence for local cohomology then yields
\begin{align*} \cdots \longrightarrow \  H^{n-1}_{I_1}(M)\,\oplus\, &H^{n-1}_{I_2}(M)  \ \longrightarrow  \ H^{n-1}_J(M) \\
 &\longrightarrow H^n_{I_1+I_2}(M) \  \longrightarrow \ H^n_{I_1}(M)\,\oplus\, H^n_{I_2}(M) \ \longrightarrow \cdots
\end{align*}
and the result follows because $I_1+I_2 = I$ and the first and last of the four terms shown are 0, since $I_1,\,I_2$ both have only $n-2$ generators.
\end{proof}

We now establish another case where Question~\ref{mainqu} has an affirmative answer.

\begin{theorem} Let $R$ be a pure $R $-submodule of a domain $S$ that is an $R$-algebra. 
Let $I$ be an ideal of $R$, and $c=\mathrm{cd}(I,R)$.
If $\mathrm{Ann}_R(H^c_{IS}(S))=0$, and, in particular, if  $\mathrm{Ann}_S(H^c_{IS}(S))=0$ 
then $\mathrm{Ann}_R(H^c_I(R))=0$.

In particular, the conclusion holds when $S$ is a regular 
domain containing a field and $R$ is pure in $S$.
\end{theorem}
\begin{proof} Since $R \to S$ is pure, it remains injective when we tensor over $R$ with $H_I^c(R)$, and so we
have an injection $H^c_I(R) \to H^c_I(S) \cong H^c_{IS}(S)$; therefore, the latter is nonzero.  
But any nonzero $r \in R$ that kills $H^c_I(R)$ will
also kill the nonzero module $H^c_I(S) \cong H^c_I(R) \otimes_R S$, contradicting the hypothesis.  

It remains to justify the statements made in the regular case. 
	If $K$ has characteristic  $p>0$, then the theorem then follows from \cite[Lemma~2.2]{HuKo}.
If $K$ has characteristic~$0$, we may use the fact that every nonzero local cohomology module with coefficients in $S$ is faithful \cite[Corollary~3.6]{Lyu1}.  

Alternatively, both cases follow from the basic theorems of Lyubeznik \cite{Lyu1, Lyu2}:   if  $S$ is regular and contains a field, we may
see that any nonzero local cohomology module $M= H^i_J(S)$ of $S$ is faithful as follows.
Localize at a minimal prime of the support of $M$, which produces a nonzero local cohomology module supported only at the
maximal ideal of an equicharacteristic regular local ring, and so isomorphic with a nonzero finite direct 
sum of copies of the injective hull of the residue field, and, consequently, faithful over the localization of $S$ and therefore
over $S$.
\end{proof}

\section{Persistence of cohomological dimension}\label{perscd}

As a consequence of Lemma~\ref{cd-Ass}, if $R\to S$ is a ring homomorphism, and $I$ is an ideal of $R$, then 
$\cdim(I,R)\geq \cdim(IS,S)$. It is easy to see that equality holds if $R$ is a direct summand of $S$ as an $R$-module, but the inequality is strict in general. Likewise, the arithmetic rank of an ideal can decrease when passing to a larger algebra. We pose the following conjecture.

\begin{conjecture}\label{conjecture-persist-cd}
	If $R$ is a Noetherian domain, and $S$ is a solid $R$-algebra, then for every ideal $I$ of $R$, we have $\mathrm{cd}(I,R)=\mathrm{cd}(IS,S)$.
\end{conjecture}

\begin{remark}
	Conjecture~\ref{conjecture-persist-cd} holds under the stronger assumption that $S$ is a module-finite $R$-algebra. Indeed, let $c=\mathrm{cd}(I,R)$. By Gruson's theorem (see, e.g., \cite[Corollary~4.3]{Vas}), since $S$ is a faithful $R$-module, $H^c_{IS}(S) \cong S \otimes_R H^c_I(R) \neq 0$, since $H^c_I(R)\neq 0$.
\end{remark}

We observe that if Question~\ref{mainqu} has a positive answer for a ring $R$ and ideal $I$, then Conjecture~\ref{conjecture-persist-cd} holds for $R$, $I$. Indeed, if $\phi:S\to R$ is a nonzero $R$-linear map, so that $\phi(s)=r\neq 0$ for some $r\in R, s\in S$, let $\phi'=\phi \circ (\cdot s)$, so that $\phi'|_{R}=\cdot r$, the map of multiplication by $r$. Applying the functor $H^c_I(-)$, we obtain that the nonzero map of multiplication by $r$ on $H^c_I(R)$ factors through $H^c_{IS}(S)$, which must then be nonzero.

Thus, as a corollary of Theorem~\ref{maintheorem}, we obtain the following special case of Conjecture~\ref{conjecture-persist-cd}.

\begin{corollary}\label{pers.cor}
	If $R$ is a domain of positive characteristic, and $I$ is an ideal such that $\mathrm{cd}(I,R)=\mathrm{ara}(I,R)=c$, then for any solid $R$-algebra $S$, we have $\mathrm{cd}(IS,S)=\mathrm{ara}(IS,S)=c$.
\end{corollary}

\section*{Acknowledgements}
The authors thank the anonymous referee for suggesting the version~(\ref{qu-b}) of Question~\ref{mainqu}. We also thank Elo\'isa Grifo and Anurag K.~Singh for catching typos in an earlier version of this manuscript.

\quad\bigskip

\end{document}